\documentclass{article}
\usepackage[T1]{fontenc}
\usepackage[a4paper,margin=1in]{geometry}
\usepackage{amsmath,amssymb,amsthm,mathtools,mathrsfs}
\usepackage{comment}
\usepackage{booktabs}
\usepackage{tabularx}
\usepackage{xcolor}
\usepackage{booktabs}
\usepackage{enumitem}
\usepackage[hidelinks]{hyperref}
\usepackage{microtype}
\usepackage{amsthm}
\usepackage{cite}
\usepackage[hidelinks]{hyperref}
 
\allowdisplaybreaks
\numberwithin{equation}{section}

\newtheorem{theorem}{Theorem}[section]
\newtheorem{proposition}[theorem]{Proposition}
\newtheorem{lemma}[theorem]{Lemma}

\theoremstyle{definition}
\newtheorem{assumption}[theorem]{Assumption}
\theoremstyle{remark}
\newtheorem{remark}[theorem]{Remark}

\newcommand{\dd}{\,\mathrm{d}}

\title{Regularity  Effects   on the Small-Diffusion Asymptotics of the Principal Eigenvalue: Non-convergence and Convergence}
\author{
Kexin Zhang\\
School of Mathematics, Sun Yat-sen University\\
Guangzhou 510275, China\\
\texttt{kxzmath@163.com}
}

\date{}

\begin{document}
\maketitle

\begin{abstract}
We study the asymptotic behavior as \(d\to0^+\) of the principal eigenvalue \(\lambda(d)\) of the one-dimensional elliptic problem
\[
\begin{cases}
-d\phi''(x)-2\alpha m'(x)\phi'(x)+c(x)\phi(x)
=\lambda(d)\phi(x), & x\in(a,b),\\
\phi'(a)=\phi'(b)=0,
\end{cases}
\]
where \(\alpha>0\) is fixed, \(m\in C^{1,1}([a,b])\), and \(c\in C([a,b])\). We assume that \(m'=0\)  only at \(x_0\in(a,b)\), where \(m\) attains a strict minimum. Within this class, we construct a counterexample $m\in C^{1,1}$ for which \(\lambda(d)\) has two distinct subsequential limits as \(d\to0^+\), leading to the non-convergence of $\lambda(d)$. We then consider another class of $C^{1,1}$ advection
functions for which $m'(x)/(x-x_0)$ has positive
one-sided limits at $x_0$.
For this class, we prove that $\lambda(d)$ converges
as $d\to0^+$. Thus, \(C^{1,1}\) regularity alone does not ensure convergence of the principal eigenvalue, even when \(m\) has exactly one strict interior  critical minimal point.
\end{abstract}

\section{Introduction}
\label{sec:introduction}
In this paper, we consider the principal eigenvalue problem
\begin{equation}
\label{eq:original-eigenproblem}
\begin{cases}
-d\phi''(x)-2\alpha m'(x)\phi'(x)+c(x)\phi(x)=\lambda(d)\phi(x),
&x\in  (a,b),\\[1mm]
\phi'(a)=\phi'(b)=0.
\end{cases}
\end{equation}
Here, $ (a,b)$ is a bounded interval, $d>0$ is the diffusion coefficient,
and $\alpha>0$ is a fixed advection coefficient. We assume that the advection term $m\in C^{1,1}([a,b])$ and linear term $c\in C([a,b])$.
We denote the principal eigenvalue by $\lambda(d)$. In this paper, we focus on the asymptotic behavior of the principal eigenvalue $\lambda(d)$ when $d$ tends to 0, especially the convergence and non-convergence properties.

It is well known that for given $d,\alpha>0,$ and functions $m(x)$ and $c(x)$, there exists a principal eigenvalue $\lambda(d)$ of the problem \eqref{eq:original-eigenproblem}. Moreover, $\lambda(d)$ has an associated principal eigenfunction, which can be chosen positive.

\subsection{Existing results}
Principal eigenvalue of the second-order elliptic operator plays an important
role in the qualitative analysis of elliptic and parabolic equations \cite{CL,PZ,LLZ,Bcmp, W75,PZhao,N10,N09,XZ26,XZ26sd,LLPZx,LLPZd,HNR2011,CCL6,book}. 
The dependence of the principal eigenvalue on the diffusion and advection coefficients has been extensively studied.

The large-advection limit $\alpha\to\infty$, with $d>0$ fixed, has also been extensively studied. For $m\in C^2$, Chen and Lou \cite{CL} established convergence under non-degeneracy assumptions on the critical points of \(m\). Peng and Zhou \cite{PZ} further studied the effects of degenerate advection and boundary conditions in one dimension. 
More recently, Bai et al. and the present author \cite{bwxzz} constructed examples
for which the principal eigenvalue does not converge as the
advection coefficient tends to infinity. In their construction,
$m$ is constant on a nontrivial interval and has infinitely
many local extrema outside this interval, accumulating at
its endpoints.

We now turn to the small-diffusion limit, which is the focus
of this paper. We consider $m\in C^{1,1}([a,b])$ whose only
critical point is a strict interior minimum $x_0\in(a,b)$, with
\[
m'(x)<0\quad\text{on }[a,x_0),\qquad
m'(x)>0\quad\text{on }(x_0,b].
\]
To place our results in context, we recall below the existing
results relevant to this configuration. The cited works
themselves allow more general critical point structures.

For the Dirichlet problem with \(c\equiv0\), Friedman and his
collaborators \cite{FA7273,DEF73,DF78} considered H\"older continuous drifts $m'$ satisfying
\begin{align}\label{grownth}
    C^{-1}|x-x_0|^\gamma\le |m'(x)|\le C|x-x_0|^\gamma,
\quad 0<\gamma\le1.
\end{align}
Their estimates imply that \(\lambda_{\mathrm D}(d)\to+\infty\)
when \(0<\gamma<1\), while for \(\gamma=1\) they give only
boundedness of \(\lambda_{\mathrm D}(d)\) as \(d\to0\).
We note that the non-degenerate \(C^2\) case \(m''(x_0)>0\)
is a subclass of this \(\gamma=1\) condition, whereas a general
\(C^{1,1}\) function need not satisfy the lower bound above.

Under Neumann boundary conditions, Chen and Lou \cite{CL12}
obtained the precise small-diffusion limit when \(m\in C^2\)
and its critical points are non-degenerate. For a unique interior
minimum \(x_0\), their local contribution is
\[
c(x_0)+2\alpha m''(x_0).
\]
For Robin boundary conditions with \(m\in C^2\), Peng, Zhang
and Zhou \cite{PZZ19} studied the effects of the boundary
condition and the Robin coefficient on the asymptotic behavior
of the principal eigenvalue. More recently, for the Dirichlet problem, Guo, Lou and Zhang
\cite{Guo25} investigated refined small-diffusion asymptotics
of the principal eigenpair under \(m\in C^2\) and structural
assumptions allowing degeneracies.

The small-diffusion results reviewed above for Neumann and
Robin boundary conditions assume $m\in C^2$. For the
Dirichlet problem, the lower-regularity estimates for the principal eigenvalue of
Friedman et al. give only boundedness under the   condition \eqref{grownth} with $\gamma=1$. Meanwhile, the
large-advection counterexample in \cite{bwxzz} involves
infinitely many local extrema accumulating at the boundary
of a constant region.

In this paper, we construct $m\in C^{1,1}([a,b])$ and
$c\in C([a,b])$ such that the principal eigenvalue of the
Neumann problem does not converge as $d\to0^+$, even though
$m$ has exactly one critical point, a strict interior minimum.

\subsection{Main results}
\label{sec:main-results}
Before stating our main result,  we first
 define for $q_-,q_+>0$:
\[
V(y)=
\begin{cases}
\alpha^2q_-^2y^2+\alpha q_-,&y<0,\\[1mm]
\alpha^2q_+^2y^2+\alpha q_+,&y>0.
\end{cases}
\]
and
\begin{align}\label{def:mu}
    \mu_\alpha(q_-,q_+)
=
\inf_{u\in\mathcal Q\setminus\{0\}}
\frac{\int_{\mathbb R}
\bigl(|u'|^2+V(y)u^2\bigr)\,dy}
{\int_{\mathbb R}u^2\,dy},
\end{align}
where
$\mathcal Q=\{u\in H^1(\mathbb R):yu\in L^2(\mathbb R)\}$. 
We impose the following assumption on $m$.

\medskip
\noindent\textbf{(A)}
The function $m$ has a unique interior critical point
$x_0\in(a,b)$, which is a strict local minimum, and satisfies
$m'(a)<0<m'(b)$.
\medskip

Our main result can be organized as follows:
\begin{theorem}\label{main}
   Fix $\alpha>0$. Consider the   function
$m\in C^{1,1}([a,b])$  satisfying  \textbf{\textup{(A)}}. Then the following statements hold for
problem \eqref{eq:original-eigenproblem}.
    \par
    (1). For any $c\in C([a,b])$, and $m$ moreover satisfies 
    \begin{equation}\label{asss}
        \lim_{x\to x_0^-}\frac{m'(x)}{x-x_0}=q_->0,\quad \lim_{x\to x_0^+}\frac{m'(x)}{x-x_0}=q_+>0,
    \end{equation}
    then \begin{equation}\label{eq:main-limit}
\begin{aligned}
\lim_{d\to0^+}\lambda(d)
=\min \{&
c(a),\,c(b),\,
c(x_0)+\mu_\alpha(q_-,q_+)\}.
\end{aligned}
\end{equation}

(2). There exist a counterexample function $m\in C^{1,1}([a,b])$  satisfying  \textbf{\textup{(A)}} and a function
$c\in C([a,b])$ such that the corresponding principal
eigenvalue $\lambda(d)$ does not converge as $d\to0^+$.
\end{theorem}
This theorem shows that both convergence and non-convergence can occur as $d\to0^+$, even for $C^{1,1}$ drift functions whose only critical point is a strict interior minimum.

The counterexample in part~(2) preserves the $C^{1,1}$
regularity of $m$ and its unique interior critical point,
but has different rescaled local profiles near $x_0$.
For such conterexample, there exist two sequences
$d_n,\tilde d_n\to0^+$ (as $n\to \infty$) such that
$$
\lim_{n\to \infty}\lambda(d_n)\neq \lim_{n\to \infty}
\lambda(\tilde d_n).
$$
Thus, different spatial scales near the single minimum
can produce distinct  limits of the
principal eigenvalue, as $d\to 0$.

Part (1) complements this example by providing a sufficient
condition for convergence in terms of the one-sided
linearizations of $m'$ at $x_0$.
When  $q_-=q_+=q>0$, Proposition~\ref{prop:equal-curvature}
gives $\mu_\alpha(q,q)=2\alpha q$. It is the same as the results for   $m\in C^2$, see \cite{CL12,PZZ19,Guo25}.
For $q_-\neq q_+$,
Proposition~\ref{prop:positive-bounds} yields
$$
2\alpha\min\{q_-,q_+\}
<
\mu_\alpha(q_-,q_+)
<
2\alpha\max\{q_-,q_+\}.
$$

\begin{remark}
The same arguments yield an analogous limit formula
when $m$ has finitely many interior critical points,
all of which are strict local minima or maxima,
provided that $m'(a)<0<m'(b)$ and \eqref{asss} holds at every
local minimum. The limit is obtained by taking the
minimum of the contributions from all interior extrema
and the endpoints. We do not pursue this extension here.
\end{remark}

The main ideas of the proof are as follows.
For part (1), we obtain the limit formula by combining
a local rescaling near $x_0$ with variational upper and
lower bounds and the contributions from the endpoints.

For part (2), we alternate the quadratic profiles $x^2$
and $2x^2$, which have different second derivatives at $0$,
in nested neighborhoods of the origin.
We use cubic connection functions to ensure that the limiting
function  belongs to $C^{1,1}$ and has exactly one critical
point. We also choose $c$ appropriately to exclude
endpoint contributions. We show that  
principal eigenvalue associated with this limit function has two distinct
subsequential limits as $d\to0^+$, by applying a continuity denpendence lemma, see Lemma \ref{cdm}.
The construction is illustrated schematically in
Figure \ref{fig:m123_even}.
\begin{figure}[htb]
	\centering
	\includegraphics[width=\linewidth]{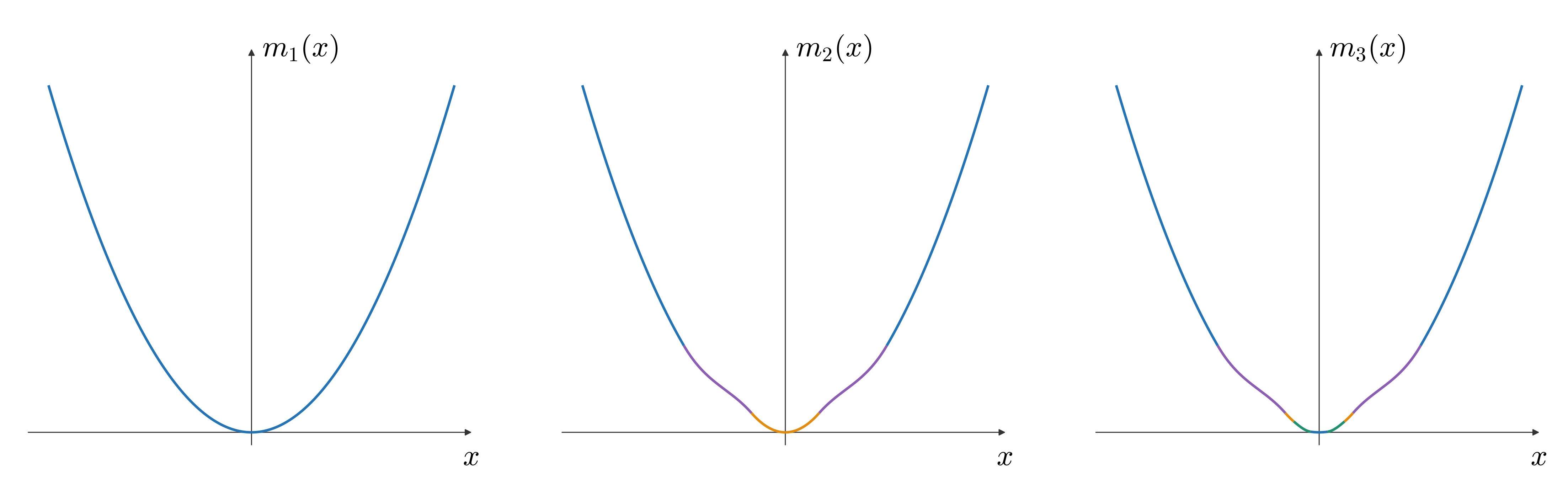}
     \caption{The construction of $m_1,m_2,m_3$}
    \label{fig:m123_even}
    \smallskip
    {\footnotesize 
    The blue line represents the funtcion $x^2$. The orange line represents the funtcion $2x^2$. The purple line and the green line represent the two connection functions, \(P_1(x)\) and \(P_2(x)\), respectively.}
\end{figure}

The remainder of this paper is organized as follows.
In Section~2, we present the variational characterizations
of $\lambda(d)$ and establish some basic properties of
$\mu_\alpha(q_-,q_+)$.
In Section~3, we prove the convergence result in
part (1) of Theorem \ref{main} by deriving matching upper
and lower bounds.
Section~4 is devoted to the construction of the
counterexample and the proof of the non-convergence result
in part (2) of Theorem \ref{main}.
Finally, Appendix \ref{app} establishes the existence of
a positive principal eigenfunction   of
the corresponding eigenvalue for the auxiliary problem
on $\mathbb R$.

\section{Preliminaries}
In this section, we give some preliminaries for preparation.
\subsection{The variational characterization of $\lambda(d)$}
The principal eigenvalue $\lambda(d)$ of the principal eigenvalue problem \eqref{eq:original-eigenproblem} admits the variational
characterization
\begin{align}\label{cc}
    \lambda(d)
=
\min_{\substack{\phi\in H^1(a,b)\\ \phi\not\equiv0}}
\frac{
\int_a^b e^{\frac{2\alpha m(x)}d}
\left(d|\phi'(x)|^2+c(x)\phi^2(x)\right)\,dx
}{
\int_a^b e^{\frac{2\alpha m(x)}d}\phi^2(x)\,dx
}.
\end{align}
We introduce the transformation
$w(x):=e^{\frac{\alpha m(x)}d}\phi(x).$
It follows that
\begin{equation}
\label{eq:variational-characterization}
\lambda(d)
=
\min_{\substack{w\in H^1(a,b)\\w\not\equiv0}}
\frac{
\int_a^b
\left[
d\left|
w'-\frac{\alpha}{d}m'w
\right|^2
+cw^2
\right]dx
}{
\int_a^b w^2\,dx
}.
\end{equation}
We will use these two kind of variational characterizations of $\lambda(d)$ for multiple times.

\subsection{Properties of $\mu_\alpha(q_-,q_+)$}

 We state that the constant
$\mu_\alpha(q_-,q_+)$ is attained and is the principal
eigenvalue of the problem
\begin{equation}\label{eq:eigenvalue-problem}
\left\{
\begin{aligned}
&-u''(y)+V(y)u(y)=\mu u(y)
&&\text{for a.e. }y\in\mathbb R,\\
&0<u\in\mathcal Q\cap H^2_{\mathrm{loc}}(\mathbb R).
\end{aligned}
\right.
\end{equation}
By the weighted decay estimate
Lemma \ref{lemwd}
and integration by parts separately on $(-\infty,0)$ and
$(0,\infty)$, we obtain
\begin{align*}
\int_{\mathbb R}\left(|u'|^2+V(y)u^2\right)\,dy
=
\int_{-\infty}^0
\left|u'-\alpha q_-yu\right|^2\,dy+
\int_0^\infty
\left|u'-\alpha q_+yu\right|^2\,dy.
\end{align*}
Therefore, $\mu_\alpha(q_-,q_+)$ admits the equivalent
variational characterization
\begin{align}\label{xx}
    \mu_\alpha(q_-,q_+)
=
\inf_{\substack{u\in\mathcal Q\\u\neq0}}
\frac{
\int_{-\infty}^0
\left|u'-\alpha q_-yu\right|^2\,dy
+
\int_0^\infty
\left|u'-\alpha q_+yu\right|^2\,dy
}{
\int_{\mathbb R}u^2\,dy
}.
\end{align}
The detailed  proof can be found in Appendix \ref{app}.

We present several properties of $\mu_\alpha(q_-,q_+)$ in the following.  

\begin{proposition}\label{prop:equal-curvature}
If $q_-=q_+=q>0$, then
$\mu_\alpha(q,q)=2\alpha q.$
\end{proposition}

\begin{proof}
When $q_-=q_+=q$, the variational characterization becomes
\[
\mu_\alpha(q,q)
=
\inf_{u\in\mathcal Q\setminus\{0\}}
\frac{
\int_{\mathbb R}
\left(
|u'|^2+\alpha^2q^2y^2u^2+\alpha q u^2
\right)dy
}{\int_{\mathbb R}u^2\,dy
}.
\]
We first prove the lower bound. For every $u\in\mathcal Q$, with $\int_{\mathbb R}u^2\,dy=1.$
By using Lemma \ref{lemwd}, we have $yu^2(y)\to0, \text{ as } |y|\to\infty,$
integration by parts gives 
\begin{align*}
0
\le\int_{\mathbb R}|u'+\alpha|q|yu|^2=
\int_{\mathbb R}
\left(
|u'|^2+\alpha^2q^2y^2u^2
\right)dy
-
\alpha |q|\int_{\mathbb R}u^2(y)\,dy.
\end{align*}
Hence, we immediately obtain
$
\mu_\alpha(q,q)\ge2\alpha q.
$

It suffices to obtain the reverse inequality. We define
$G_q(y)=
e^{-\frac{\alpha|q|}{2}y^2}.$
Then $G_q\in\mathcal Q\setminus\{0\}$.
Substituting $G_q$ into the preceding identity gives $
\mu_\alpha(q,q)\le 2\alpha q.$
\end{proof}

\begin{proposition}
\label{prop:positive-bounds}
Suppose that $q_->0$, $q_+>0$, and $q_-\neq q_+$. Then
\begin{equation}\label{eq:positive-bounds}
2\alpha\min\{q_-,q_+\}
<
\mu_\alpha(q_-,q_+)
<
2\alpha\max\{q_-,q_+\}.
\end{equation}
\end{proposition}

\begin{proof}
We first claim that for
$I=(-\infty,0)$ or $(0,\infty)$ and let $q\neq 0$. Then, for every
$u\in\mathcal Q$,
\begin{equation*}\label{eq:half-line-factorization}
\int_I |u'-\alpha qyu|^2\,dy
=
\int_I |u'+\alpha |q|yu|^2\,dy
+
\alpha(|q|+q)\int_Iu^2\,dy.
\end{equation*}
Indeed,
\[
|u'-\alpha qyu|^2
=
|u'+\alpha |q|yu|^2
-\alpha(q+|q|)y(u^2)'.
\]
Using the decay property Lemma \ref{lemwd}:
\[
yu^2(y)\to 0
\quad\text{as } |y|\to\infty,
\]
and integrating by parts on each half-line, we obtain
\eqref{eq:half-line-factorization}. In particular,
\begin{equation}\label{eq:half-line-positive}
\int_I |u'-\alpha qyu|^2\,dy
\ge
\begin{cases}
2\alpha q\displaystyle\int_Iu^2\,dy,
& q>0,\\[2mm]
0,
& q<0.
\end{cases}
\end{equation}

To prove this proposition, without loss of generality, assume that
$q_-<q_+.$
Let $u_1$ be the positive minimizer associated with
$\mu_\alpha(q_-,q_+)$, normalized by
$
\int_{\mathbb R}u_1^2\,dy=1.
$
Set
\[
M_-:=\int_{-\infty}^0u_1^2\,dy,
\quad
M_+:=\int_0^\infty u_1^2\,dy.
\]
Since $u_1>0$ on $\mathbb R$, we have
\[
M_->0,\quad M_+>0,\quad M_-+M_+=1.
\]
Applying \eqref{eq:half-line-positive} on the two half-lines gives
\begin{align*}
\mu_\alpha(q_-,q_+)
\ge
2\alpha q_-M_-+2\alpha q_+M_+
=
2\alpha q_-+
2\alpha(q_+-q_-)M_+>
2\alpha q_-.
\end{align*}
This proves the strict lower bound.

To prove the upper bound, we consider the test function
\[
G(y):=
\begin{cases}
e^{-\frac{\alpha q_-}{2}y^2},
&y<0,\\
e^{-\frac{\alpha q_+}{2}y^2},
&y\ge0.
\end{cases}
\]
Since $G$ and $G'$ are continuous at $0$, we have
$G\in\mathcal Q$. 
It follows from \eqref{xx} that
\begin{align*}
\mu_\alpha(q_-,q_+)\le \frac{
\int_{-\infty}^0
\left|G'-\alpha q_-yG\right|^2\,dy
+
\int_0^\infty
\left|G'-\alpha q_+yG\right|^2\,dy
}{
\int_{\mathbb R}G^2\,dy
}=
2\alpha\sqrt{q_-q_+}<
2\alpha q_+.
\end{align*}
Therefore,
\[
2\alpha q_-
<
\mu_\alpha(q_-,q_+)
<
2\alpha q_+,
\]
which proves \eqref{eq:positive-bounds}.
\end{proof}

\section{Proof of (1) of Theorem \ref{main}}\label{sec3}

We give the proof by using the upper bound estimate and the lower bound estimate for the principal eigenvalue $\lambda(d)$ respectively.
Throughout this section, we assume that
$m\in C^{1,1}([a,b])$ satisfies \textbf{\textup{(A)}} and
\[
\lim_{x\to x_0^-}\frac{m'(x)}{x-x_0}=q_->0,
\quad
\lim_{x\to x_0^+}\frac{m'(x)}{x-x_0}=q_+>0.
\]
\begin{lemma}
\label{lem:upper-bound}
Let $\lambda(d)$ be the principal eigenvalue of
the principal eigenvalue problem \eqref{eq:original-eigenproblem}, where $\alpha>0$, $c\in C([a,b])$, and $m\in C^{1,1}([a,b])$ satisfies
Assumption \textbf{\textup{A}}. Then we have 
\begin{equation*}\label{eq:main-limit}
\limsup_{d\to 0}\lambda_d
\le
\min \{c(a),
c(b),
c(x_0)+\mu_{\alpha}(q_-,q_+)
\},
\end{equation*}
where we define $\mu_{\alpha}(q_-,q_+)$ in \eqref{def:mu}.
\end{lemma}
\begin{proof}
For the interior minimal point $x_0$,
we first aim to prove the upper bound estimate
\begin{align}\label{imsps}
\limsup_{d\to0}\lambda(d)\le c(x_0)+\mu_\alpha(q_-,q_+).
\end{align}

Without loss of generality, we assume that $x_0=0$ and  $a<0<b$. 
From  $x_0=0$, we have 
\[
m'(0)=0,
\quad
\lim_{x\to0^-}\frac{m'(x)}{x}=q_-,
\quad
\lim_{x\to0^+}\frac{m'(x)}{x}=q_+.
\]

Define
\[
A(y):=
\begin{cases}
q_-y,&y<0,\\
q_+y,&y>0.
\end{cases}
\]
Let $u_1\in\mathcal Q$ be the positive principal eigenfunction
associated with $\mu_\alpha(q_-,q_+)$, normalized by
$\int_{\mathbb R}u_1^2\,dy=1.$
Thus
\[
-u_1''+V(y)u_1
=\mu_\alpha(q_-,q_+)u_1,
\]
where
\[
V(y)=\alpha^2A^2(y)+\alpha A'(y)
=
\begin{cases}
\alpha^2q_-^2y^2+\alpha q_-,&y<0,\\
\alpha^2q_+^2y^2+\alpha q_+,&y>0.
\end{cases}
\]

For $\delta>0$ small, define
\[
\eta_\delta(y):=(u_1(y)-\delta)_+.
\]
Since $u_1(y)\to0$ as $|y|\to\infty$, the function
$\eta_\delta$ has compact support. Moreover,
$\eta_\delta\in H^1(\mathbb R)$ and
\[
\eta_\delta'(y)
=u_1'(y)\mathbf 1_{\{u_1>\delta\}}(y)
\quad\text{for a.e. }y\in\mathbb R.
\]

For sufficiently small $d>0$, define
\[
w_{d,\delta}(x)
:=
d^{-1/4}\eta_\delta\left(\frac{x}{\sqrt d}\right),
\qquad x\in(a,b).
\]
Since $\eta_\delta$ has compact support,  for all sufficiently small $d$, we have 
\[
w_{d,\delta}\in H_0^1(a,b)
\subset H^1(a,b).
\]

It follows from the variational characterization that
\begin{align}
\lambda(d)
&\le\frac{
\int_{a}^b
d\left|
w_{d,\delta}'-\frac{\alpha}{d}m'w_{d,\delta}
\right|^2
+cw_{d,\delta}^2
dx
}{
\int_{a}^bw_{d,\delta}^2\,dx
}.\nonumber\\
&\overset{x=\sqrt{d}y}{=}
\frac{
\int_{\mathbb R}
\left|\eta_\delta'-\alpha \frac{m'(\sqrt{d}y)}{\sqrt{d}}\eta_\delta\right|^2
+c(\sqrt d\,y)\eta_\delta^2
dy
}{
\int_{\mathbb R}\eta_\delta^2\,dy
}.
\nonumber\\
&=
\frac{
\int_{\mathbb R}
\left|\eta_\delta'-\alpha A_d(y)\eta_\delta\right|^2
+c(\sqrt d\,y)\eta_\delta^2
dy
}{
\int_{\mathbb R}\eta_\delta^2\,dy
},
\label{eq:upper-test-delta}
\end{align}
where we denote $A_d(y):=\frac{m'(\sqrt d\,y)}{\sqrt d}$ in the second equality.

We note that for each fixed $\delta>0$, $\eta_\delta$ has compact support.
The assumptions on $m'$ imply that
$A_d(y)\to A(y), \text{ as } d \to 0,$
uniformly on compact subsets of $\mathbb R$. Moreover,
$
c(\sqrt d y)\to c(0), \text{ as } d \to 0,
$
uniformly on compact subsets of $\mathbb R$. Hence, letting
$d\to0$ in \eqref{eq:upper-test-delta}, we obtain
\begin{align}
\limsup_{d\to0}\lambda(d)
\le
\frac{
\int_{\mathbb R}
\left[
\left|\eta_\delta'-\alpha A(y)\eta_\delta\right|^2
+c(0)\eta_\delta^2
\right]dy
}{
\int_{\mathbb R}\eta_\delta^2\,dy
}=\frac{
\int_{\mathbb R}
|\eta_\delta'|^2+V(y)\eta_\delta^2
\,dy
}{
\int_{\mathbb R}\eta_\delta^2\,dy
}+c(0).
\label{eq:upper-after-d}
\end{align}
In the second equality, we use the fact that 
 $\eta_\delta^2\in H^1(\mathbb R)$ has compact support on $\mathbb R$ and the 
integration by parts separately on $(-\infty,0)$ and $(0,\infty)$ due to $A(0^-)=A(0^+)=0.$

We finally let $\delta\to0$ in \eqref{eq:upper-after-d}. Since $u_1>0$, we have
\[
|u_1-\eta_\delta|
=\min\{u_1,\delta\}\le u_1.
\]
Note that $u_1\in H^1(\mathbb R)$, Lebesgue dominated theorem implies that 
\[
\eta_\delta\to u_1
 \text{ in }L^2(\mathbb R) \ \text{ as } \delta\to 0.
\]
Moreover, since
$\eta_\delta'
=u_1'\mathbf 1_{\{u_1>\delta\}}$,
we also have
\[
\eta_\delta'\to u_1'
 \text{ in }L^2(\mathbb R) \ \text{ as } \delta\to 0.
\]
We note that $yu_1\in L^2(\mathbb R)$, then Lebesgue dominated theorem implies that 
\[
y\eta_\delta\to yu_1
 \text{ in }L^2(\mathbb R)\ \text{ as }  \delta\to 0.
\]
Hence, as $\delta\to 0$,  we obtain
\[
\frac{
\int_{\mathbb R}
\left(|\eta_\delta'|^2+V(y)\eta_\delta^2\right)dy
}{
\int_{\mathbb R}\eta_\delta^2\,dy
}
\to
\frac{
\int_{\mathbb R}
\left(|u_1'|^2+V(y)u_1^2\right)dy
}{
\int_{\mathbb R}u_1^2\,dy
}
=
\mu_\alpha(q_-,q_+).
\]
Letting $\delta\to0$ in \eqref{eq:upper-after-d}, we derive 
\[
\limsup_{d\to0}\lambda(d)
\le c(0)+\mu_\alpha(q_-,q_+).
\]
Hence, \eqref{imsps} holds true.

Next, we consider upper bound effects the boundary points $x=a$ and $x=b$ and we aim to prove that 
 \begin{align}\label{bd
 }
 \limsup_{d\to 0}\lambda(d)\le \min\{c(a),c(b)\},
 \end{align} 
 and hence, combining with \eqref{imsps},  we complete the proof.

We prove the assertion at $a$; the proof at $b$ is symmetric.  Assume $m'(a)<0$.  Choose $\delta>0$ and $\eta>0$ such that
\[
m'(x)\le-\eta
\qquad (a\le x\le a+2\delta).
\]
Let $\chi\in C^1([a,b])$ satisfy
\[
0\le\chi\le1,
\quad
\chi\equiv1\ \text{on }[a,a+\delta],
\quad
\chi\equiv0\ \text{on }[a+2\delta,b], \quad|\chi'|\le \frac{C}{\delta}.
\]
Define
\begin{equation}\label{eq:endpoint-trial}
v_d(x)
:=
\chi(x)
\exp\!\left(\frac{\alpha(m(x)-m(a))}{d}\right).
\end{equation}

We insert the test function $v_d(x)$ into  the  variational characterization of $\lambda_d$, and  direct calculation yields that
\begin{align}
\lambda(d)
&\le\frac{
\int_{a}^b
d\left|
v_{d}'-\frac{\alpha}{d}m'v_{d}
\right|^2
+cv_{d}^2
dx
}{
\int_{a}^b v_{d}^2\,dx
}\nonumber\\
&=\frac{
\int_{a}^{a+2\delta}
d \chi'^2(x)
\exp\left(\frac{2\alpha(m(x)-m(a))}{d}\right)
+cv_{d}^2
dx
}{
\int_{a}^b v_{d}^2\,dx
}\nonumber\\
&\le  \frac{
\int_{a+\delta}^{a+2\delta}
d \chi'^2(x)
\exp\left(\frac{2\alpha(m(x)-m(a))}{d}\right)
dx
}{
\int_{a}^b v_{d}^2\,dx
}+\max_{x\in [a,a+2\delta]}c(x)\nonumber\\
&:=I_d+\max_{x\in [a,a+2\delta]}c(x).
\label{eqbdup}
\end{align}

We next estimate $I_d$.
Since $m'(a)<0$, there exists $\gamma>0$ such that when $a+\delta\le x\le a+2\delta$, we have
\[
m(x)\le m(a)-\gamma.
\]
For $d<\frac{\delta}{2}$, we note that
$m(a+2d)-m(a)\ge -Kd$ for some positive $K$. Then we have
\begin{equation*}\label{eq:endpoint-den-lower}
\int_{a}^b v_{d}^2\,dx
\ge
\int_{a+d}^{a+2d}
e^{\frac{2\alpha(m(x)-m(a))}{d}}\dd x
\ge
\int_{a+d}^{a+2d}
e^{\frac{2\alpha(m(a+2d)-m(a))}{d}}\dd x
\ge
d e^{-2\alpha K}.
\end{equation*}
Then, we obtain
\begin{align}\label{ID}
    I_d= \frac{
\int_{a}^{a+2\delta}
d \chi'^2(x)
\exp\left(\frac{2\alpha(m(x)-m(a))}{d}\right)
dx
}{
\int_{a}^b v_{d}^2\,dx
}\nonumber\le \frac{
\int_{a+\delta}^{a+2\delta}
\frac{Cd}{\delta^2}
\exp\left(\frac{2\alpha(m(x)-m(a))}{d}\right)
dx
}{
d e^{-2\alpha K}
}\le \frac{\frac{Cd}{\delta} e^{-2\alpha\gamma/d}}{d e^{-2\alpha K}}.
\end{align}

We recall \eqref{eqbdup}:
\begin{align}
    \lambda(d)
\le I_d+\max_{x\in [a,a+2\delta]}c(x)\le \frac{\frac{Cd}{\delta} e^{-2\alpha\gamma/d}}{d e^{-2\alpha K}}+\max_{x\in [a,a+2\delta]}c(x).
\end{align}
Let $d\to 0$ first and then $\delta \to 0$, we obtain the upper bound estimate  $\limsup_{d\to 0}\lambda(d)\le c(a)$. Similarly, we obtain the upper bound estimate for $x=b$ and then
$\limsup_{d\to0}\lambda(d)\le \min\{ c(a), c(b)\}.$
\end{proof}

Before we give the lower bound estimate,
we first give a local lower bound estimate of the non-critical points.
\begin{lemma}\label{lem:weighted-noncritical-estimate}
Assume that $m\in C^1$, and let $J=(r,s)\subset(a,b)$. Suppose that, for some $\delta>0$,
$|m'(x)|\ge\delta$ for
$x\in[r,s]$.
Then, for every $w\in H_0^1(r,s)$,
\begin{equation}\label{eq:weighted-noncritical-estimate}
\int_a^bd\left|w'-\frac{\alpha }{d}m'w\right|^2\,dx\ge \frac{\alpha^2\delta^2}{d} \int_a^bw^2\,dx,\quad w\in H^1_0(r,s).
\end{equation}
\end{lemma}

\begin{proof}
Since $m'$ is continuous and does not vanish on the connected
interval $[r,s]$, it has a fixed sign.
Assume first that
\[
m'(x)\ge\delta,
\qquad x\in[r,s].
\]
A direct calculation gives
\[
\left(e^{\frac{2\alpha m(x)}{d}}\right)'
=
\frac{2\alpha}{d}m'(x)
e^{\frac{2\alpha m(x)}{d}}
\ge
\frac{2\alpha\delta}{d}
e^{\frac{2\alpha m(x)}{d}}>0.
\]
It follows that
\begin{align*}
\frac{2\alpha\delta}{d}
\int_r^s
e^{\frac{2\alpha m(x)}{d}}u^2(x)\,dx
&\le
\int_r^s
\left(e^{\frac{2\alpha m(x)}{d}}\right)'
u^2(x)\,dx=
e^{\frac{2\alpha m(x)}{d}}u^2(x)
|_r^s
-
2\int_r^s
e^{\frac{2\alpha m(x)}{d}}
u(x)u'(x)\,dx,
\end{align*}
where  $w=e^{\frac{\alpha m}{d}}u$.
Since $w\in H_0^1(r,s)$, we have $u(r)=u(s)=0$, and hence
\begin{align}
   \frac{2\alpha\delta}{d}
\int_r^s
e^{\frac{2\alpha m(x)}{d}}u^2(x)\,dx
&\le
-2\int_r^s
e^{\frac{2\alpha m(x)}{d}}
u(x)u'(x)\,dx\nonumber\\
&\le 
2
\left(
\int_r^s
e^{\frac{2\alpha m(x)}{d}}u^2(x)\,dx
\right)^{1/2}
\left(
\int_r^s
e^{\frac{2\alpha m(x)}{d}}|u'(x)|^2\,dx
\right)^{1/2}, 
\end{align}
where we use the Cauchy-Schwarz inequality in the last inequality.
Therefore we have
\[
\frac{\alpha\delta}{d}
\left(
\int_r^s
e^{\frac{2\alpha m(x)}{d}}u^2(x)\,dx
\right)^{1/2}
\le
\left(
\int_r^s
e^{\frac{2\alpha m(x)}{d}}|u'(x)|^2\,dx
\right)^{1/2}.
\]
Squaring both sides and multiplying by $d$, we obtain
\[
d\int_r^s
e^{\frac{2\alpha m(x)}{d}}|u'(x)|^2\,dx
\ge
\frac{\alpha^2\delta^2}{d}
\int_r^s
e^{\frac{2\alpha m(x)}{d}}u^2(x)\,dx.
\]
We note that $w=e^{\frac{\alpha m}{d}}u$, equivalently,  we have
$$\int_a^bd\left|w'-\frac{\alpha }{d}m'w\right|^2\,dx\ge \frac{\alpha^2\delta^2}{d} \int_a^bw^2\,dx,\quad w\in H^1_0(r,s).$$

If
\[
m'(x)\le-\delta,
\qquad x\in[r,s],
\]
then
\[
-\left(e^{\frac{2\alpha m(x)}{d}}\right)'
=
-\frac{2\alpha}{d}m'(x)
e^{\frac{2\alpha m(x)}{d}}
\ge
\frac{2\alpha\delta}{d}
e^{\frac{2\alpha m(x)}{d}}>0.
\]
Applying the same argument to
$-\left(e^{2\alpha m(x)/d}\right)'$ gives
\eqref{eq:weighted-noncritical-estimate}.
\end{proof}

Next, we give a local lower bound  estimate for the unique minimal critical point $x=x_0$.
Define
\begin{align*}
    E[w]=\int_a^b
\left[
d\left|w'-\frac{\alpha}{d}m'w\right|^2
+cw^2
\right]dx.
\end{align*}

\begin{lemma}\label{lem:lower-x0}
For every $\varepsilon>0$, there exists $\delta>0$  such that for
every
$w\in H_0^1(x_0-\delta,x_0+\delta),$
one has
\begin{equation}\label{eq:lower-x0-local}
E[w]
\ge
\left(c(x_0)+\mu_\alpha(q_-,q_+)-C\varepsilon\right)
\|w\|_{L^2(a,b)}^2,
\end{equation}
for some positive constant $C$.
Here $w=0$  outside the interval $(x_0-\delta,x_0+\delta)$.
\end{lemma}

\begin{proof}
Let
 $q_0:=\min\{q_-,q_+\}>0,
q_1:=\max\{q_-,q_+\}.$
Fix $0<\varepsilon<\frac12 q_0$. Without loss of generality, we assume that $x_0=0$.
By the assumptions on $m'(x)$ and the continuity of $c(x)$, there exists
$\delta=\delta(\varepsilon)>0$ such that
\begin{align*}
    \left|m'(x)-q_-x\right|
\le \varepsilon|x|,
\quad -\delta<x<0,
\quad \left|m'(x)-q_+x\right|
\le \varepsilon|x|,
\quad 0<x<\delta,
\end{align*}
and
\[
c(x)\ge c(0)-\varepsilon,
\qquad |x|<\delta.
\]

For $w\in H_0^1(-\delta,\delta)$, we introduce the rescaled variables
\[
y:=\frac{x}{\sqrt d},
\quad
z(y):=d^{1/4}w(\sqrt d\,y),
\]
and let
\[
J_{d,\delta}
:=
\left(
-\frac{\delta}{\sqrt d},
\frac{\delta}{\sqrt d}
\right).
\]
Since $w\in H_0^1( -\delta,\delta)$, we have
$z\in H_0^1(J_{d,\delta})$. We extend $z$ by zero outside
$J_{d,\delta}$ and regard it as an element of $H^1(\mathbb R)$. Then we obtain
\[
\int_{-\delta}^{\delta}w^2(x)\,dx
=
\int_{\mathbb R}z^2(y)\,dy.
\]
We recall that
$A_d(y)
:=
\frac{m'(\sqrt d\,y)}{\sqrt d}
$
and
\[
A(y):=
\begin{cases}
q_-y,&y<0,\\
q_+y,&y>0.
\end{cases}
\]
Moreover, the choice of $\delta$ implies
\begin{equation}
\label{eq:Ad-A-relative-bound}
|A_d(y)-A(y)|
\le \varepsilon|y|,
\qquad y\in J_{d,\delta}.
\end{equation}

Since for almost everywhere 
$y\in\mathbb R$, we have 
$A(y)^2\geq q_0^2y^2,
A'(y)\geq 0$,
and  $A$ is continuous at $0$,
 we use integration by parts to the interval $(-\infty,0)$ and $(0,+\infty)$ gives
\begin{align}
\int_{\mathbb R}|z'-\alpha Az|^2\,dy
&=
\int_{\mathbb R}
\left(
|z'|^2+\alpha^2A^2z^2+\alpha A'z^2
\right)\,dy \geq
\alpha^2q_0^2
\int_{\mathbb R}y^2z^2\,dy.
\end{align}
Consequently,
\begin{equation}
\label{eq:weighted-coercivity-general}
\int_{\mathbb R}y^2z^2\,dy
\leq
\frac{1}{\alpha^2q_0^2}
\int_{\mathbb R}|z'-\alpha Az|^2\,dy.
\end{equation}

By the variational characterization of
$\mu_\alpha(q_-,q_+)$,
\begin{align}\label{vc}
    \int_{\mathbb R}|z'-\alpha Az|^2\,dy
\ge
\mu_\alpha(q_-,q_+)
\int_{\mathbb R}z^2\,dy.
\end{align}
Now, we can obtain the following estimate for the energy functional $E[w]$ by selecting $\varepsilon$ small:
\begin{align}
    E[w]&=\int_{-\delta}^{\delta}d\left|
w'(x)-\frac{\alpha}{d}m'(x)w(x)
\right|^2+c(x)w^2(x)dx\nonumber\\
&\ge\int_{\mathbb R}
\left|z'(y)-\alpha A_d(y)z(y)\right|^2+(c(0)-\varepsilon)z^2(y)dy\nonumber\\
&=\int_{\mathbb R}
\left|(z'(y)-\alpha A(y)z(y))-\alpha z(y)( A_d(y)-A(y))\right|^2+(c(0)-\varepsilon)z^2(y)dy\nonumber\\
&\ge(1-\theta)\int_{\mathbb R}
|z'(y)-\alpha A(y)z(y)|^2dy-\frac{\alpha^2}{\theta}\int_{\mathbb R}| A_d(y)-A(y)|^2z^2(y)dy+(c(0)-\varepsilon)\int_{\mathbb R}z^2(y)dy\nonumber\\
&\ge(1-\theta)\int_{\mathbb R}
|z'(y)-\alpha A(y)z(y)|^2dy-\frac{\alpha^2\varepsilon^2}{\theta}\int_{\mathbb R}y^2z^2(y)dy+(c(0)-\varepsilon)\int_{\mathbb R}z^2(y)dy\nonumber\\
&\ge(1-\theta-\frac{\varepsilon^2}{\theta q_0^2})\int_{\mathbb R}
|z'(y)-\alpha A(y)z(y)|^2dy+(c(0)-\varepsilon)\int_{\mathbb R}z^2(y)dy\nonumber\\
&\ge \left[\left(1-\theta-\frac{\varepsilon^2}{q_0^2\theta}\right)\mu_\alpha(q_-,q_+) +(c(0)-\varepsilon)\right]\int_{\mathbb R}
z^2(y) dy\nonumber\\
&=\left[\left(1-\theta-\frac{\varepsilon^2}{q_0^2\theta}\right)\mu_\alpha(q_-,q_+) +(c(0)-\varepsilon)\right]\int_{-\delta}^{\delta}w^2(x)\,dx.
\end{align}
In the second inequality, we use the elementary inequality: 
for any $\theta\in(0,1)$, it  holds that
\[
|X-Y|^2
\ge
(1-\theta)|X|^2-\frac{1}{\theta}|Y|^2.
\]
We use the estimate \eqref{eq:Ad-A-relative-bound} in the third inequality.
In the forth inequality, we use the inequality \eqref{eq:weighted-coercivity-general} and in the fifth
 inequality, we use  \eqref{vc}. Finally, let $\theta=\frac{\varepsilon}{ q_0}$, we obtain
\begin{align*}
   E[w]&\ge \left[(1-\frac{2\varepsilon}{ q_0})\mu_\alpha(q_-,q_+)+(c(0)-\varepsilon)\right]\int_{-\delta}^{\delta}w^2(x)\,dx\nonumber\\
&\ge (\mu_\alpha(q_-,q_+) +c(0)-C\varepsilon)\|w\|_{L^2(a,b)}^2. 
\end{align*}
\end{proof}

Finally, we are at a position to  obtain a global lower bound estimate of $\lambda(d)$.

\begin{lemma}
\label{lem:lowr-bound}
Let $\lambda(d)$ be the principal eigenvalue of
the principal eigenvalue problem \eqref{eq:original-eigenproblem}, where $\alpha>0$, $c\in C([a,b])$, and $m$ satisfies
Assumption \textbf{\textup{A}}. Then we have 
\begin{equation*}\label{eq:main-limit}
\liminf_{d\to 0}\lambda(d)
\ge
\min\{c(a),
c(b),
c(x_0)+\mu_{\alpha}(q_+,q_-)
\},
\end{equation*}
where we define $\mu_{\alpha}(q_+,q_-)$ in \eqref{def:mu}.
\end{lemma}
\begin{proof}
    Let $w$ be the principal eigenfunction of  the principal eigenvalue $\lambda(d)$ normalized by \(\int_a^b w^2\,dx =1\).
 Choose \(R>0\)  small such that
 \[
c(x)\ge c(a)-\varepsilon,
\quad a\le x<a+R,
\quad \text{and}\quad
c(x)\ge c(b)-\varepsilon,
\quad b-R<x\le b.
\]
 and $((a,a+R)\cup(b-R,b))\cap B(x_0,R)=\emptyset$ 
where $B(x_0,R)$=$(x_0-R,x_0+R)$.

Set
\[
I_1:=\left(a+\frac{R}{2},b-\frac{R}{2}\right)
\setminus
\overline{B(x_0,R/2)}.
\]
Let $I_a:=[a,a+R),
I_b:=(b-R,b]
$.
 Then 
\[
\left\{
I_a,I_b,I_1,
B(x_0,R)
\right\},
\]
is an open covering of \([a,b]\),where the openness is understood
in the relative topology of \([a,b]\).

Let
$\{\zeta_a^2,\zeta_b^2,\zeta_1^2,\zeta_0^2,\}$
be a smooth partition of unity subordinate to this covering. Thus
\(\zeta_i\in C^\infty([a,b])\), \(0\leq \zeta_i\leq 1\) for $i=0,1$ and
\[
\zeta_a^2(x)+\zeta_b^2(x)+ \zeta_1^2(x)+\zeta_0^2(x)=1,
\quad \forall x\in [a,b],
\]
with
\[
\operatorname{supp}\zeta_a\subset[a,a+R),\quad
\operatorname{supp}\zeta_b\subset(b-R,b],\quad
\operatorname{supp} \zeta_1\subset I_1,
\quad
\operatorname{supp} \zeta_0\subset B(x_0,R),
\]
Moreover, the functions can be chosen so that
\[
|\zeta_0'|,|\zeta_1'|\leq \frac{C}{R},
\]
for some constant \(C>0\) independent of \(R\).  Each point \(x\in [a,b]\) belongs to at most two sets in the covering; consequently, at most two of the functions \(\zeta_i\) ($i=a,b,0,1$) are nonzero at any point. Hence,
\begin{equation}\label{gradient bound}
(\zeta_a'(x))^2+(\zeta_b'(x))^2+ |\zeta_1(x)'|^2+|\zeta_0(x)'|^2
\leq\frac{C}{R^2},
\quad \forall x\in [a,b].
\end{equation}

For $W_a=\zeta_aw$ and $W_b=\zeta_bw$, we have the following estimate:
\begin{align}
   & E[W_a]
\ge
\int_a^b c(x)W_a^2\,dx
\ge
(c(a)-\varepsilon)\int_a^b W_a^2\,dx\\
 & E[W_b]
\ge
\int_a^b c(x)W_b^2\,dx
\ge
(c(b)-\varepsilon)\int_a^b W_b^2\,dx.
\end{align}

For \(W_0 = \zeta_0 w\), we have $W_0\in H_0^1(B(x_0,R))$.
Thanks to Lemma \ref{lem:lower-x0}, we obtain
\[
E(W_0) \ge (\mu_\alpha(q_-,q_+) +c(x_0)-C\varepsilon)\int_{a}^{b}W_0^2(x)\,dx.
\]
Moreover, define $W_1 = \zeta_1 w$, by direct calculus, we have the following equality:
\begin{align}
  &E(W_a)+E(W_b)+   E(W_1)+   E(W_0)\nonumber\\
  =& \int_a^b d(\zeta_a'^2+\zeta_b'^2+\zeta_1'^2+\zeta_0'^2)w^2+(\zeta_a^2+\zeta_b^2+\zeta_1^2+\zeta_0^2)(d\left|w'-\frac{\alpha}{d}m'w\right|^2+cw^2)\,dx\nonumber\\
  &+\int_a^b d[(\zeta_a^2)'+(\zeta_b^2)'+(\zeta_1^2)'+(\zeta_0^2)']\left(w'-\frac{\alpha}{d}m'w\right)w\,dx\nonumber\\
     =&\int_a^b d(\zeta_a'^2+\zeta_b'^2+ \zeta_1'^2+ \zeta_0'^2)w^2\,dx+E[w].
\end{align}
where we use $\zeta_a^2+\zeta_b^2+ \zeta_1^2+ \zeta_0^2=1$ and $((\zeta_a^2)'+(\zeta_b^2)'+(\zeta_1^2)'+(\zeta_0^2)')=(\zeta_a^2+\zeta_b^2+\zeta_1^2+\zeta_0^2)'=0$.

Next, we give the estimate of  $E[W_1]$ with $W_1=\zeta_1w$. We
 note that on $I_1$, there is a constant $\delta>0$ such that $|m'(x)|>\delta$.  Thanks to Lemma \ref{lem:weighted-noncritical-estimate}, by dividing  $I_1$ into several finite intervals, we have 
\begin{align}
 E[W_1]\ge \left(\frac{\alpha^2\delta^2}{d} +c_*\right)\int_a^bW_1^2\,dx.
\end{align}

Therefore, we note that $w$ is the principal eigenfunction with $\int_a^b w^2\, dx=1$, and then we have
\begin{align}\label{111}
 \lambda(d)&= E[w]=  E(W_a)+E(W_b)+    E(W_1)+    E(W_0)-\int_a^b d(\zeta_a'^2+\zeta_b'^2+\zeta_1'^2+\zeta_0'^2)w^2\,dx\nonumber\\
&\ge  (c(a)-\varepsilon)\int_a^b W_a^2\,dx+ (c(b)-\varepsilon)\int_a^b W_b^2\,dx+(\mu_\alpha(q_-,q_+) +c(x_0)-C\varepsilon)\int_{a}^{b}W_0^2\,dx\nonumber\\
&+ E[W_1]-\int_a^b d(\zeta_a'^2+\zeta_b'^2+\zeta_1'^2+\zeta_0'^2)w^2\,dx\nonumber\\
&\ge  (c(a)-\varepsilon)\int_a^b W_a^2\,dx+ (c(b)-\varepsilon)\int_a^b W_b^2\,dx+(\mu_\alpha(q_-,q_+) +c(x_0)-C\varepsilon)\int_{a}^{b}W_0^2\,dx\nonumber\\
&+ \left(\frac{\alpha^2\delta^2}{d} +c_*\right)\int_a^bW_1^2\,dx-\frac{Cd}{R^2}\int_a^b w^2\,dx\nonumber\\
&\ge(\Lambda-C\varepsilon) \int_a^b \left(W_a^2+W_b^2+ W_1^2+ W_0^2\right)\,dx-\frac{Cd}{R^2}\nonumber\\
&=(\Lambda-C\varepsilon) \int_a^b w^2\,dx-\frac{Cd}{R^2},
\end{align}
where $\Lambda=\min\{c(a), c(b),\mu_\alpha(q_-,q_+) +c(x_0)\}$.
In the last inequality, we choose $d$ sufficiently small such that $\left(\frac{\alpha^2\delta^2}{d} +c_*\right)>\Lambda$, and we use the fact that $w^2=W_a^2+W_b^2+W_1^2+W_0^2$. Sending $d\to 0$ first and then $\varepsilon\to 0$, we obtain the lower bound estimate of $\lambda(d)$.
\end{proof}

 \textbf{Proof of (1) of Theorem \ref{main}}: Combining with Lemma \ref{lem:upper-bound} and Lemma \ref{lem:lowr-bound}, we finish (1) of Theorem \ref{main}.

\section{Proof of  (2) of Theorem \ref{main}}
In this section, we construct the counterexample in the interval $(-1,1)$, and complete the proof of  (2) of Theorem \ref{main}.

Let \(I=(-1,1)\). For \(m\in C^1(\overline I)\), define
\[
\lambda(d,m)
=
\inf_{\phi\in H^1(I),\,\phi\not\equiv0}
\frac{
\int_I e^{\frac{2\alpha}{d}m(x)}
\left(d|\phi'(x)|^2+c(x)\phi^2(x)\right)\,dx
}{ \int_I e^{\frac{2\alpha}{d}m(x)}\phi^2(x)\,dx
}.
\]
Assume that
\[
c_*=\min_{\overline I}c(x)>0,
\qquad
C^*=\max_{\overline I}c(x).
\]
Before we construct the counterexample, we first give a continuous dependence lemma.
\begin{lemma}\label{cdm}
For every fixed \(d>0\) and every \(m_1,m_2\in C(\overline I)\), we have
\[
\left|\lambda(d,m_1)-\lambda(d,m_2)\right|
\le
C^*
\left(
e^{\frac{4\alpha}{d}\|m_1-m_2\|_{L^\infty(I)}}-1
\right).
\]
\end{lemma}

\begin{proof}
Set
\[
\delta=\|m_1-m_2\|_{L^\infty(I)}.
\]
Then
\[
e^{-\frac{2\alpha}{d}\delta}
e^{\frac{2\alpha}{d}m_2}
\le
e^{\frac{2\alpha}{d}m_1}
\le
e^{\frac{2\alpha}{d}\delta}
e^{\frac{2\alpha}{d}m_2}.
\]
Hence, for every \(\phi\in H^1(I)\), \(\phi\not\equiv0\),
\[
\frac{
\int_I e^{\frac{2\alpha}{d}m_1}
\left(d|\phi'|^2+c\phi^2\right)\,dx
}{
\int_I e^{\frac{2\alpha}{d}m_1}\phi^2\,dx
}
\le
e^{\frac{4\alpha}{d}\delta}
\frac{
\int_I e^{\frac{2\alpha}{d}m_2}
\left(d|\phi'|^2+c\phi^2\right)\,dx
}{
\int_I e^{\frac{2\alpha}{d}m_2}\phi^2\,dx
}.
\]
Taking the infimum over \(\phi\) gives
\[
\lambda(d,m_1)\le
e^{\frac{4\alpha}{d}\delta}\lambda(d,m_2).
\]
Since the constant test function gives
\[
\lambda(d,m_2)\le C^*,
\]
we obtain
\[
\lambda(d,m_1)-\lambda(d,m_2)
\le
C^*
\left(
e^{\frac{4\alpha}{d}\delta}-1
\right).
\]
Interchanging \(m_1\) and \(m_2\) yields the desired estimate.
\end{proof}

\subsection{The construction of the counterexample}
To construct a counterexample for which the principal eigenvalue fails to converge, we impose the following assumption on \(c\) in order to exclude possible contributions from the boundary points.
\begin{assumption}\label{assc}
The function $c$ satisfies
\begin{equation*}\label{assc}
    c(-1),\,c(1)>c(0)+9\alpha.
\end{equation*}
\end{assumption} 
We now construct the counterexample such that the corresponding principal eigenvalue does not converge as $d\to 0$.   In the following, we only construct the part  on \([0,1]\), and we extend it evenly to \([-1,1]\).

\textbf{Step 1.}
Let
\[
m_1(x)=x^2.
\]
By the known small-diffusion limit and (1) of Theorem \ref{main} under Assumption \ref{assc} for $c(x)$,
\[
\lim_{d\to0}\lambda(d,m_1)=c(0)+4\alpha.
\]
Choose \(0<d_1<\frac{1}{4\alpha}\ln(1+\frac{1}{2C^*})\)  such that
\[
\left|\lambda(d_1,m_1)-\bigl(c(0)+4\alpha\bigr)\right|<\frac12
\]

Set
\[
\tau_1=d_1^2.
\]
Since \(m_1(x)=x^2\), we have
\[
|m_1(x)|\le \tau_1,
\quad 0\le x\le d_1.
\]

Define \(m_2\) by
\[
m_2(x)=
\begin{cases}
2x^2, & 0\le x\le \frac{d_1}{3},\\
P_{1}(x), & \frac{d_1}{3}<x<d_1,\\
m_1(x), & d_1\le x\le 1,
\end{cases}
\]
where
\[
P_{1 }(x)
=
\frac{2}{9}d_1^2
+
\int_{d_1/3}^{x}
\left(
\frac{27}{4d_1}t^2-8t+\frac{13}{4}d_1
\right)\,dt.
\]
Then we have
\[
P_{1 }\left(\frac{d_1}{3}\right)=\frac{2}{9}d_1^2,
\
P_{1 }(d_1)=d_1^2,\
P_{1 }'\left(\frac{d_1}{3}\right)=\frac{4}{3}d_1,
\
P_{1}'(d_1)=2d_1, 
\]
Moreover, we have $P_{1}'(x)>0$ and $|P_{1}''(x)|\le \frac{11}{2},\ \frac{d_1}{3}<x<d_1.$
Therefore we have $m_2(x)\in C^1[0,1]$ and  $\text{Lip}(m_2'(x))\le \frac{11}{2}$. Then we obtain 
\[
m_2\in C^{1,1}([0,1]).
\] 
Moreover,
\[
\|m_2-m_1\|_{L^\infty(0,1)}\le d_1^2.
\]
By Lemma \ref{cdm} and \(0<d_1<\frac{1}{4\alpha}\ln(1+\frac{1}{2C^*})\), we have
\[
\left|\lambda(d_1,m_2)-\lambda(d_1,m_1)\right|
\le
C^*\left(e^{4\alpha d_1}-1\right)
<\frac12.
\]
Hence
\[
\left|\lambda(d_1,m_2)-\bigl(c(0)+4\alpha\bigr)\right|<1.
\]

\textbf{Step 2.}
For  \(m_2(x)\), (1) of Theorem \ref{main} and Assumption \ref{assc} gives
\[
\lim_{d\to0}\lambda(d,m_2)=c(0)+8\alpha.
\]
Choose \(0<d_2<\min\{\frac{\sqrt2}{3}d_1,
\frac{1}{4\alpha}\ln(1+\frac{1}{3C^*})\}\) such that
\[
\left|\lambda(d_2,m_2)-\bigl(c(0)+8\alpha\bigr)\right|<\frac13,
\]

Set
\[
\tau_2=d_2^2.
\]
Since \(m_2(x)=2x^2\) near \(0\), we have
\[
m_2(x)\le d_2^2,
\quad
0\le x\le \frac{d_2}{\sqrt2}.
\]

Define \(m_3\) by
\[
m_3(x)=
\begin{cases}
x^2, & 0\le x\le \dfrac{d_2}{3\sqrt2},\\
P_2(x), & \frac{d_2}{3\sqrt2}<x<\frac{d_2}{\sqrt2},\\
m_2(x), & \dfrac{d_2}{\sqrt2}\le x\le 1,
\end{cases}
\]
where
\[
P_2(x)
=
\frac{d_2^2}{18}
+
\int_{d_2/(3\sqrt2)}^{x}
\left(
-\frac{27}{2\sqrt2d_2}t^2
+14t
-\frac{13}{4}\frac{d_2}{\sqrt2}
\right)\,dt.
\]
Then
\[
P_2\left(\frac{d_2}{3\sqrt2}\right)=\frac{d_2^2}{18},
\quad
P_2\left(\frac{d_2}{\sqrt2}\right)=d_2^2,\quad
P_2'\left(\frac{d_2}{3\sqrt2}\right)=\frac{2d_2}{3\sqrt2},
\quad
P_2'\left(\frac{d_2}{\sqrt2}\right)=\frac{4d_2}{\sqrt2}.
\]
Moreover, we have $P_2'(x)>0$ and $|P_2''(x)|\le \frac{19}{2}$ for $\frac{d_2}{3\sqrt2}<x<\frac{d_2}{\sqrt2}$.
Therefore
\[
m_3\in C^{1,1}([0,1]).
\]
 Moreover,
\[
\|m_3-m_2\|_{L^\infty(0,1)}\le d_2^2.
\]
Thus, due to Lemma \ref{cdm}, we have
\[
\left|\lambda(d_2,m_3)-\lambda(d_2,m_2)\right|
\le
C^*\left(e^{4\alpha d_2}-1\right)
<\frac13.
\]
Consequently,
\[
\left|\lambda(d_2,m_3)-\bigl(c(0)+8\alpha\bigr)\right|
<\frac23.
\]

We now continue the above construction inductively.

Suppose first that $m_{2n-1}\in C^{1,1}$ has been constructed. By the previous
step, there exists a neighborhood of $0$ in which
\[
m_{2n-1}(x)=x^2.
\]

\textbf{Step $2n-1$.}
We have
\[
m_{2n-1}(x)=x^2,
\quad
0\le x\le \frac{d_{2n-2}}{3\sqrt2}.
\]
By (1) of Theorem \ref{main} and Assumption \ref{assc}, we have
\[
\lim_{d\to0}\lambda(d,m_{2n-1})
=
c(0)+4\alpha.
\]
Choose
\[
0<d_{2n-1}
<
\min\left\{
\frac{d_{2n-2}}{3\sqrt2},
\frac{1}{4\alpha}
\ln\left(1+\frac{1}{2nC^*}\right)
\right\}
\]
such that
\[
\left|
\lambda(d_{2n-1},m_{2n-1})
-
\bigl(c(0)+4\alpha\bigr)
\right|
<
\frac{1}{2n}.
\]

Set
\[
\tau_{2n-1}=d_{2n-1}^2,
\]
we have
\[
m_{2n-1}(x)=x^2\le d_{2n-1}^2,
\quad
0\le x\le d_{2n-1}<\frac{d_{2n-2}}{3\sqrt2}.
\]

Define \(m_{2n}\) by
\[
m_{2n}(x)=
\begin{cases}
2x^2, & 0\le x\le \frac{d_{2n-1}}{3},\\
P_{2n-1}(x), & \frac{d_{2n-1}}{3}<x<d_{2n-1},\\
m_{2n-1}(x), & d_{2n-1}\le x\le 1,
\end{cases}
\]
where
\[
P_{{2n-1}}(x)
=
\frac{2}{9}d_{2n-1}^2
+
\int_{d_{2n-1}/3}^{x}
\left(
\frac{27}{4d_{2n-1}}t^2-8t+\frac{13}{4}d_{2n-1}
\right)\,dt.
\]
Then we have
\[
P_{{2n-1}}\left(\frac{d_{2n-1}}{3}\right)=\frac{2}{9}d_{2n-1}^2,
\
P_{{2n-1}}(d_{2n-1})=d_{2n-1}^2,\]
\[
P_{{2n-1}}'\left(\frac{d_{2n-1}}{3}\right)=\frac{4}{3}d_{2n-1},
\
P_{{2n-1}}'(d_{2n-1})=2d_{2n-1}, 
\]
Moreover, we have $|P_{{2n-1}}''(x)|\le \frac{11}{2},\ \frac{d_{2n-1}}{3}<x<d_{2n-1}.$
Therefore we have $m_{2n}'(x)\in C[0,1]$ and  $\text{Lip}(m_{2n}'(x))\le \frac{19}{2}$. Then we obtain 
\[
m_{2n}\in C^{1,1}([0,1]).
\] 
Moreover,
\[
\|m_{2n}-m_{2n-1}\|_{L^\infty(0,1)}\le d_{2n-1}^2.
\]
By Lemma \ref{cdm}, we have
\[
\left|\lambda(d_{2n-1},m_{2n})-\lambda(d_{2n-1},m_{2n-1})\right|
\le
C^*\left(e^{4\alpha d_{2n-1}}-1\right)
<\frac{1 }{2n}.
\]
Hence
\begin{align}\label{odd}
\left|\lambda(d_{2n-1},m_{2n})-\bigl(c(0)+4\alpha\bigr)\right|<\frac1n.
\end{align}

\textbf{Step $2n$.}
Since
\[
m_{2n}(x)=2x^2,
\quad
0\le x\le \frac{d_{2n-1}}3,
\]
(1) of Theorem \ref{main}  and Assumption \ref{assc} gives
\[
\lim_{d\to0}\lambda(d,m_{2n})
=
c(0)+8\alpha.
\]
Choose
\[
0<d_{2n}
<
\min\left\{
\frac{\sqrt2}{3}d_{2n-1},
\frac{1}{4\alpha}
\ln\left(1+\frac{1}{(2n+1)C^*}\right)
\right\}
\]
such that
\[
\left|
\lambda(d_{2n},m_{2n})
-\bigl(c(0)+8\alpha\bigr)
\right|
<
\frac{1}{2n+1}.
\]

Set
\[
\tau_{2n}=d_{2n}^2,
\]
we have
\[
m_{2n}(x)=2x^2\le d_{2n}^2,
\quad
0\le x\le\frac{d_{2n}}{\sqrt2}< \frac{d_{2n-1}}{3}.
\]

Define
\[
m_{2n+1}(x)=
\begin{cases}
x^2,
&0\le x\le \dfrac{d_{2n}}{3\sqrt2},\\[1ex]
P_{2n}(x),
&\dfrac{d_{2n}}{3\sqrt2}
 <x<
 \dfrac{d_{2n}}{\sqrt2},\\[1ex]
m_{2n}(x),
&\dfrac{d_{2n}}{\sqrt2}\le x\le1,
\end{cases}
\]
where
\[
P_{2n}(x)
=
\frac{d_{2n}^2}{18}
+
\int_{d_{2n}/(3\sqrt2)}^x
\left(
-\frac{27}{2\sqrt2d_{2n}}t^2
+14t
-\frac{13d_{2n}}{4\sqrt2}
\right)\,dt.
\]
A direct calculation gives
\[
P_{2n}\left(\frac{d_{2n}}{3\sqrt2}\right)
=
\frac{d_{2n}^2}{18},
\quad
P_{2n}\left(\frac{d_{2n}}{\sqrt2}\right)
=
d_{2n}^2,
\]
and
\[
P_{2n}'\left(\frac{d_{2n}}{3\sqrt2}\right)
=
\frac{2d_{2n}}{3\sqrt2},
\quad
P_{2n}'\left(\frac{d_{2n}}{\sqrt2}\right)
=
\frac{4d_{2n}}{\sqrt2}.
\]
Therefore the values and first derivatives match at both
junction points.

In particular, \(P_{2n}'(x)>0\) on the connecting interval.
Consequently,
\[
m_{2n+1}\in C^{1,1}([0,1]),
\quad
\operatorname{Lip}(m_{2n+1}')\le\frac{19}{2}.
\]
Furthermore,
\[
m_{2n+1}(x)=x^2,
\quad
0\le x\le\frac{d_{2n}}{3\sqrt2},
\]
which provides the required starting property for the next
inductive step.

Since  we have
\[
\|m_{2n+1}-m_{2n}\|_{L^\infty(0,1)}
\le d_{2n}^2.
\]
It follows from Lemma~\ref{cdm} that
\[
\begin{aligned}
&\left|
\lambda(d_{2n},m_{2n+1})
-\lambda(d_{2n},m_{2n})
\right|\le
C^*\left(e^{4\alpha d_{2n}}-1\right)
<
\frac{1}{2n+1}.
\end{aligned}
\]
Consequently,
\begin{align}\label{even}
   \left|
\lambda(d_{2n},m_{2n+1})
-\bigl(c(0)+8\alpha\bigr)
\right|
<
\frac{2}{2n+1}. 
\end{align}

Under this construction, we obtain the function sequence
\(\{m_n\}_{n=1}^{+\infty}\).
Moreover, by construction, for each \(n\), the function \(m_n\) satisfies
\[
m_n(0)=m_n'(0)=0,
\quad
\operatorname{Lip}(m_n')\le L:=\frac{19}{2},
\]
and
\begin{equation}\label{eq:mn-difference}
\operatorname{supp}(m_{n+1}-m_n)\subset[0,d_n],
\quad
\|m_{n+1}-m_n\|_{L^\infty(0,1)}
\le d_n^2.
\end{equation}
Here, \(\{d_n\}_{n=1}^{+\infty}\) satisfies
$
d_{2n}<\frac{\sqrt2}{3}d_{2n-1},
d_{2n+1}<\frac{1}{3\sqrt2}d_{2n},
 n\ge1.
$
Since
$
\frac{1}{3\sqrt2}<\frac{\sqrt2}{3},
$
we obtain
\begin{equation}\label{Dn-bound}
d_n<\frac{\sqrt2}{3}d_{n-1},
\quad n\ge2.
\end{equation}

\subsection{Proof of (2) of Theorem \ref{main}}

In this subsection, we prove that the limit function \(m^*\) of the
sequence \(\{m_n\}_{n=1}^{+\infty}\) constructed above is the desired
counterexample. The corresponding principal eigenvalue does not
converge as \(d\to0\).

\begin{proposition}\label{prop:limit-potential}
There exists \(m^*\in C^{1,1}([-1,1])\) such that
\[
m_n\to m^*
\quad\text{in }C^1([-1,1]).
\]
Moreover, $m^*$ has only one   minimal critical point.
\end{proposition}

\begin{proof}
We divide the proof into two steps.

\medskip
\noindent
\textbf{Step 1. Convergence in \(C^1([0,1])\).}

We note that \(m_{n+1}\) and \(m_n\) coincide outside
\([0,d_n]\). Since
$m_n'(0)=m_{n+1}'(0)=0$
and
$
\operatorname{Lip}(m_n'),
\operatorname{Lip}(m_{n+1}')\le L,
$
we obtain, for \(0<x<d_n\),
\[
\begin{aligned}
|m_{n+1}'(x)-m_n'(x)|
\le |m_{n+1}'(x)-m_{n+1}'(0)|
   +|m_n'(x)-m_n'(0)| \le 2Ld_n.
\end{aligned}
\]
Hence
\begin{equation}\label{eq:derivative-difference}
\|m_{n+1}'-m_n'\|_{L^\infty(0,1)}
\le 2Ld_n.
\end{equation}

By \eqref{Dn-bound}, we
have
\[
\sum_{n=1}^{+\infty}d_n^2<+\infty,
\quad
\sum_{n=1}^{+\infty}d_n<+\infty.
\]
Combining \eqref{eq:mn-difference} and
\eqref{eq:derivative-difference}, we obtain, for \(q>p\),
\begin{equation}\label{cha1}
\|m_q-m_p\|_{C^1([0,1])}
\le
\sum_{n=p}^{q-1}
\|m_{n+1}-m_n\|_{C^1([0,1])} \notag\le
\sum_{n=p}^{q-1}\left(d_n^2+2Ld_n\right).
\end{equation}
Therefore, the right-hand side tends to zero as
\(p\to\infty\). Thus, \(\{m_n\}\) is a Cauchy sequence in
\(C^1([0,1])\), and there exists \(m^*\in C^1([0,1])\) such that
\begin{equation*}
m_n\to m^*
\quad\text{in }C^1([0,1]).
\end{equation*}


\medskip
\noindent
\textbf{Step 2. The limit \(m^*\) belongs to \(C^{1,1}\).}

For every \(x,y\in[0,1]\), we have
\[
|m_n'(x)-m_n'(y)|
\le L|x-y|.
\]
Passing to the limit, thanks to Step 1, we obtain
\[
|(m^*)'(x)-(m^*)'(y)|
\le L|x-y|.
\]
Therefore,
\[
(m^*)'\in C^{0,1}([0,1]),
\quad
\operatorname{Lip}((m^*)')\le L,
\]
and hence
\[
m^*\in C^{1,1}([0,1]).
\]
Moreover, \(m_n'(0)=0\) for every \(n\), so
$(m^*)'(0)=0.$

We extend \(m^*\) evenly to \([-1,1]\) by setting
\[
m^*(x):=m^*(-x),
\quad -1\le x<0.
\]
Its derivative is given by
\[
(m^*)'(x)=
\begin{cases}
(m^*)'(x),&x>0,\\
0,&x=0,\\
-(m^*)'(-x),&x<0.
\end{cases}
\]
If \(x<0<y\), then
\[
\begin{aligned}
|(m^*)'(x)-(m^*)'(y)|
\le |(m^*)'(-x)|+|(m^*)'(y)|\le L(-x)+Ly=L|x-y|.
\end{aligned}
\]
Together with the same-side estimates, this proves that the even
extension satisfies
\[
m^*\in C^{1,1}([-1,1]),
\quad
\operatorname{Lip}(m^*)'\le L.
\]

It suffices to verify that $m^*$ has only one  minimal critical point.
For every fixed $x\in(0,1]$, choose $N$ so large that $d_N<x$.
Since $d_n$ is decreasing and
$m_{n+1}=m_n$ on $(d_n,1]$, all subsequent modifications
leave a neighborhood of $x$ unchanged. Hence
$
(m^*)'(x)=m_N'(x)>0.
$
By evenness, $(m^*)'(x)<0$ for $x\in[-1,0)$.
Together with $(m^*)'(0)=0$, this shows that $0$ is the
unique critical point and a strict minimum of $m^*$.
Moreover, $(m^*)'(-1)=-2$ and $(m^*)'(1)=2$.
Thus $m^*$ satisfies \textbf{\textup{(A)}} on $[-1,1]$.
Moreover, $m^*$ has only one local minimum point.
\end{proof}

\textbf{Proof of (2) of Theorem \ref{main}.}
It suffices for us to prove 
\[
\lim_{n\to\infty}\lambda(d_{2n-1},m^*)
=
c(0)+4\alpha,
\quad
\lim_{n\to\infty}\lambda(d_{2n},m^*)
=
c(0)+8\alpha.
\]

We first estimate the difference between \(m^*\) and the potential
present at the \(k\)-th stage. From \eqref{Dn-bound},
\[
\|m^*-m_{k+1}\|_{L^\infty(0,1)}
\le
\sum_{j=k+1}^\infty d_j^2\le d_k^2\sum_{j=1}^\infty \left(\frac{2}{9}\right)^j\le \frac{2d_k^2}{7}.
\]

Applying Lemma~\ref{cdm} with \(d=d_k\), we obtain
\[
\begin{aligned}
\left|
\lambda(d_k,m^*)-\lambda(d_k,m_{k+1})
\right|\le
C^*
\left(
e^{\left\{
\frac{4\alpha}{d_k}
\|m^*-m_{k+1}\|_{L^\infty(-1,1)}
\right\}}-1
\right)\le
C^*
\left(
e^{\frac{8\alpha}{7}d_k}-1
\right).
\end{aligned}
\]
Since \(d_k\to0\), it follows that
\begin{equation}\label{eq:eigenvalue-tail}
\left|
\lambda(d_k,m^*)-\lambda(d_k,m_{k+1})
\right|
\to 0.
\end{equation}

For the odd subsequence, \eqref{odd} gives
\[
\left|
\lambda(d_{2n-1},m_{2n})
-\bigl(c(0)+4\alpha\bigr)
\right|
<\frac1n.
\]
Therefore, by \eqref{eq:eigenvalue-tail},
\[
\begin{aligned}
&\left|
\lambda(d_{2n-1},m^*)
-\bigl(c(0)+4\alpha\bigr)
\right|\le
\left|
\lambda(d_{2n-1},m^*)
-\lambda(d_{2n-1},m_{2n})
\right|
+\frac1n.
\end{aligned}
\]
Hence
\[
\lim_{n\to\infty}\lambda(d_{2n-1},m^*)
=
c(0)+4\alpha.
\]

Similarly, for the even subsequence, the construction \eqref{even} gives
\[
\left|
\lambda(d_{2n},m_{2n+1})
-\bigl(c(0)+8\alpha\bigr)
\right|
<
\frac{2}{2n+1}.
\]
Using \eqref{eq:eigenvalue-tail} again, we find
\[
\begin{aligned}
\left|
\lambda(d_{2n},m^*)
-\bigl(c(0)+8\alpha\bigr)
\right|\le
\left|
\lambda(d_{2n},m^*)
-\lambda(d_{2n},m_{2n+1})
\right|
+\frac{2}{2n+1}
\to 0.
\end{aligned}
\]
Therefore,
\[
\lim_{n\to\infty}\lambda(d_{2n},m^*)
=
c(0)+8\alpha.
\]

Since \(\alpha>0\), the two limits are different. 
Consequently, \(\lambda(d,m^*)\) does not converge as \(d\to0\).

\appendix
\section{The principal  eigenvalue problem on $\mathbb{R}$}\label{app}

We shall prove that the variational quantity
$\mu_\alpha(q_-,q_+)$ is attained and is the principal
eigenvalue of the problem
\begin{equation}\label{eq:eigenvalue-problem}
\left\{
\begin{aligned}
&-u''(y)+V(y)u(y)=\mu u(y)
&&\text{for a.e. }y\in\mathbb R,\\
&0<u\in\mathcal Q\cap H^2_{\mathrm{loc}}(\mathbb R).
\end{aligned}
\right.
\end{equation}

We first give a decay property for functions in $\mathcal Q$ for preparation.

\begin{lemma}
\label{lemwd}
If $u\in\mathcal Q$, then
\[
\lim_{|y|\to\infty}|y|\,|u(y)|^2=0.
\]
\end{lemma}

\begin{proof}
We first prove that
\[
\lim_{y\to+\infty} y u(y)^2=0.
\]
Since $u\in H^1(\mathbb R)$, we take its locally absolutely
continuous representative. For any $s>t>0$, we have
\[
\begin{aligned}
s u(s)^2-t u(t)^2
=\int_t^s \bigl(yu(y)^2\bigr)'\,dy=\int_t^s u(y)^2\,dy
  +2\int_t^s yu(y)u'(y)\,dy.
\end{aligned}
\]
Therefore, by the Cauchy--Schwarz inequality,
\[
\begin{aligned}
\left|s u(s)^2-t u(t)^2\right|
&\leq
\int_t^s u(y)^2\,dy
+2\int_t^s |yu(y)|\,|u'(y)|\,dy\\
&\leq
\int_t^s u(y)^2\,dy
+2
\left(\int_t^s y^2u(y)^2\,dy\right)^{1/2}
\left(\int_t^s |u'(y)|^2\,dy\right)^{1/2}.
\end{aligned}
\]
Since
\[
u\in L^2(\mathbb R),\quad
yu\in L^2(\mathbb R),\quad
u'\in L^2(\mathbb R),
\]
the right-hand side converges to zero uniformly in $s>t$ as
$t\to+\infty$. Hence $\{yu(y)^2\}_{y>0}$ converge as $y\to+\infty$ due to  the Cauchy Theorem. Consequently, there exists $L\geq0$
such that
\[
\lim_{y\to+\infty}yu(y)^2=L.
\]

We claim that $L=0$. Suppose, to the contrary, that $L>0$.
Then there exists $R>0$ such that
\[
yu(y)^2\geq\frac{L}{2}
\quad\text{for every }y\geq R.
\]
Multiplying this inequality by $y$, we obtain
\[
y^2u(y)^2
\geq
\frac{L}{2}y
\quad\text{for every }y\geq R.
\]
Consequently,
\[
\int_R^\infty y^2u(y)^2\,dy
\geq
\frac{L}{2}\int_R^\infty y\,dy
=+\infty,
\]
which contradicts the assumption $yu\in L^2(\mathbb R)$.
Therefore, we have
$
\lim_{y\to+\infty}yu(y)^2=0.
$ Similarly, we have $
\lim_{y\to-\infty}yu(y)^2=0.
$
\end{proof}

\begin{proposition}
\label{prop:principal-eigenfunction}
Assume that \(\alpha>0\) and \(q_-,q_+> 0\). Then the infimum in
\eqref{def:mu} is attained. More precisely, there exists
$u_1\in\mathcal Q\cap H^2_{\mathrm{loc}}(\mathbb R)$ and
$\|u_1\|_{L^2(\mathbb R)}=1$,
such that
$u_1(y)>0
\text{ for every }y\in\mathbb R$, and
\[
-u_1''+V(y)u_1
=
\mu_\alpha(q_-,q_+)u_1
\quad\text{for a.e. }y\in\mathbb R.
\]
The principal eigenvalue $\mu_\alpha(q_-,q_+)$ is simple, and hence
$u_1$ is unique up to multiplication by a nonzero
constant.
\end{proposition}

\begin{proof}
We note that the functional has a non-negaive lower bound.
Choose a minimizing sequence $\{u_n\}\subset\mathcal Q$ such that
$\|u_n\|_{L^2(\mathbb R)}=1$ and
\[
\int_{\mathbb R}\bigl(|u_n'|^2+V(y)u_n^2\bigr)\,dy
\le \mu_\alpha(q_-,q_+)+\frac1n.
\]
Since
\[
V(y)\geq
\alpha^2\min\{q_-^2,q_+^2\}y^2
-\alpha\max\{|q_-|,|q_+|\},
\]
there exists a uniform constant $C$ such that 
\begin{equation}\label{compact}
 \int_{\mathbb R}  |u_n'|^2+y^2|u_n|^2\,dy\le C.
\end{equation}
Then we obtain that $\{u_n\}_{n=1}^{+\infty}$ has a uniform bound in $H^1(\mathbb{R})$ and $\{yu_n\}_{n=1}^{+\infty}$ has a uniform bound in $L^2(\mathbb{R})$. According to the Banach-Alaoglu Theorem, we could select a subsequence, denoted by itself for simplicity, such that 
\begin{align*}
u_n\rightharpoonup u_1,\quad
yu_n \rightharpoonup U_1.
\end{align*}
 We claim that $U_1=yu_1$. In fact, for any $f\in C_c^{\infty}(\mathbb R)$, we have 
$$
\int_{\mathbb R}yu_nf\,dy\to\int_{\mathbb R}U_1f\,dy,\quad n\to+\infty,
$$
due to the fact $yu_n \rightharpoonup U_1$. Since $u_n \rightharpoonup u_1$ and $yf\in C_c^{\infty}(\mathbb R)$  we have 
$$
\int_{\mathbb R}u_nyf\,dy\to\int_{\mathbb R}u_1yf\,dy,\quad n\to+\infty.
$$
The uniqueness of the weak limit gives that $U_1=yu_1\in L^2(\mathbb R)$ alomst everywhere. 

Since $yu_{n_j}\rightharpoonup yu_1$ weakly in $L^2(\mathbb R)$, the weak lower semicontinuity of the $L^2$-norm yields
\[
\|yu_1\|_{L^2(\mathbb R)}^2
\le
\liminf_{j\to\infty}
\|yu_{n_j}\|_{L^2(\mathbb R)}^2
\le C.
\]
Consequently, for every \(R>0\),
\[
\int_{\{|y|>R\}} u_1^2\,dy
\leq
\frac{1}{R^2}
\int_{\{|y|>R\}} y^2u_1^2\,dy
\leq
\frac{C}{R^2}.
\]

For any $R>0$, since $H^1(-R,R)\subset\subset L^2(-R,R)$,  by  a diagonal argument, after passing to a subsequence, still
denoted by  $\{n_j\}_{j=1}^{+\infty}$ ($n_j\to+\infty$, as $j\to+\infty$)  such that 
$$
u_{n_j}\to u_1 \text{ strongly in }L^2(-R,R).
$$

Due to \eqref{compact}, there exists a positive constant $C$ such that 
$$
R^2\int_{\{|y|>R\}}u_{n_j}^2\,dy\le\int_{\{|y|>R\}}y^2u_{n_j}^2\,dy\le \int_{\mathbb R}y^2u_{n_j}^2\,dy\le C,
$$
which means
\begin{align}
\int_{\{|y|>R\}}u_{n_j}^2\,dy\le \frac{C}{R^2}.
\end{align}

Then we claim that 
$$u_{n_j}\to u_1,\text{ strongly in }L^2(\mathbb{R}).$$
In fact,  for any $\epsilon>0$, choose $R>\frac{\sqrt{C}}{\epsilon}$, there exists $j(\epsilon)$ such that 
\begin{equation}\label{e1}
    \|u_{n_j}- u_1\|_{L^2(-R,R)}<\epsilon, \quad j\ge j(\epsilon),
\end{equation}
and 
\begin{equation}\label{e2}
    \|u_{n_j}\|_{L^2(|y|>R)}\le \frac{\sqrt{C}}{R}\le \epsilon,\quad \|u_{1}\|_{L^2(|y|>R)}\le \frac{\sqrt{C}}{R}\le \epsilon.
\end{equation}
Combining the estimates \eqref{e1} and \eqref{e2}, we have
\begin{align*}
    \|u_{n_j}- u_1\|_{L^2(\mathbb{R})}&\le \|u_{n_j}- u_1\|_{L^2(-R,R)}+\|u_{n_j}\|_{L^2(|y|>R)}+\|u_{1}\|_{L^2(|y|>R)}\le 3\epsilon,\ j\ge 
j(\epsilon),
\end{align*}
and $u_{n_j}\to u_1\text{ strongly in }L^2(\mathbb{R})$ holds true  and then $\|u_1\|_{L^2(\mathbb{R})}=1$.

Therefore, by the
  weakly lower semicontinuous of the $L^2$ norm and the strongly convergence in $L^2$, we have
\begin{align*}
\mu_\alpha(q_-,q_+)=\liminf_{n\to+\infty}\int_{\mathbb R}
\left(
|u_{n_j}'|^2+V(y)u_{n_j}^2
\right)\,dy\ge
\int_{\mathbb R}
\left(
|u_1'|^2+V(y)u_1^2
\right)\,dy\ge \mu_\alpha(q_-,q_+),
\end{align*}
and $\mu_\alpha(q_-,q_+)$ is attained by choosing the function $u_1$.

We next show that \(u_1\in H^1(\mathbb{R})\) satisfies the corresponding eigenvalue
problem in the weak sense. Let
\[
J(u)
:=
\frac{
\displaystyle
\int_{\mathbb R}
\left(
|u'|^2+V(y)u^2
\right)\,dy
}{
\displaystyle
\int_{\mathbb R}u^2\,dy
},
\qquad u\in\mathcal Q\setminus\{0\},
\]
and, for any \(\psi\in C_c^\infty(\mathbb R)\), set
$
h(t):=J(u_1+t\psi)
$.
Since \(u_1\) is a minimizer of \(J\), we have \(h'(0)=0\).
A direct calculation gives
\[
\begin{aligned}
0
&=
h'(0)\int_{\mathbb R}u_1^2\,dy\\
&=
2\int_{\mathbb R}
\left[
u_1'\psi'
+
\bigl(V(y)-\mu_\alpha(q_-,q_+)\bigr)u_1\psi
\right]\,dy\\
&=
2\left\langle
-u_1''
+
\bigl(V-\mu_\alpha(q_-,q_+)\bigr)u_1,
\psi
\right\rangle_{\mathcal D'(\mathbb R),\mathcal D(\mathbb R)} .
\end{aligned}
\]
Since this equality holds for every
\(\psi\in C_c^\infty(\mathbb R)\) in the distribution sense, it follows that 
\[
-u_1''+V(y)u_1
=
\mu_\alpha(q_-,q_+)u_1
\qquad\text{in }\mathcal D'(\mathbb R).
\]

Since \(V\in L^\infty_{\mathrm{loc}}(\mathbb R)\) and
\(u_1\in L^2_{\mathrm{loc}}(\mathbb R)\), we have
\[
\bigl(V-\mu_\alpha(q_-,q_+)\bigr)u_1
\in L^2_{\mathrm{loc}}(\mathbb R).
\]
Therefore, the above distributional equation implies that
\[
u_1''
=
\bigl(V-\mu_\alpha(q_-,q_+)\bigr)u_1
\in L^2_{\mathrm{loc}}(\mathbb R),
\]
and hence
\[
u_1\in H^2_{\mathrm{loc}}(\mathbb R).
\]
By using the one-dimensional Sobolev embedding Theorem, 
the functions \(u_1\) and \(u_1'\) are continuous on
\(\mathbb R\).

Since \(J(|u_1|)=J(u_1)\), we may assume that \(u_1\ge0\). The strong maximum principle then implies that \(u_1>0\) on \(\mathbb R\).
A standard orthogonality argument shows that
\(\mu_\alpha(q_-,q_+)\) is simple; hence  $u_1$ 
 is unique up to multiplication by a nonzero constant.
\end{proof}


\bibliographystyle{abbrv}
\bibliography{bib}
\end{document}